\documentclass[11pt]{amsart}

\usepackage{amssymb}

\usepackage{tikzexternal}
{dgs}    

\usepackage[latin1]{inputenc}
\usepackage{hyperref}

\newtheoremstyle{fancy}{}{}{\itshape}{}{\textbf\bgroup}{.\egroup}{ }{}
\newtheoremstyle{fanci}{}{}{\rm}{}{\textbf\bgroup}{.\egroup}{ }{}
\newtheoremstyle{ghost}{}{}{\itshape}{}{\textbf\bgroup}{\egroup}{ }{}

\theoremstyle{fancy}
\numberwithin{equation}{section} \swapnumbers

\newtheorem{prop}[equation]{Proposition}
\newtheorem{thm}[equation]{Theorem}

\theoremstyle{fanci}

\newtheorem{csa}[equation]{Covering Spectrum Algorithm}

\newtheorem{dfn}[equation]{Definition}

\newtheorem{exa}[equation]{Example}

\newtheorem{rem}[equation]{Remark}

\newcommand{\cref}[1]{Corollary~\ref{#1}}   

\newcommand{\Z}{\mathbb Z}  
\newcommand{\R}{\mathbb R} 
\newcommand{\ff}{\mathbb F} 
\newcommand{\nnR}{\mathbb R_{\geq 0}} 

\renewcommand{\phi}{\varphi}

\newcommand{\tM}{\widetilde{M}}
\newcommand{\tX}{\widetilde{X}}

\newcommand{\mU}{\mathcal{U}}

\newcommand{\mL}{\mathcal{L}}

\newcommand{\bs}{\backslash}

\newcommand{\CovSpec}{\operatorname{CovSpec}} 
\newcommand{\Jump}{\operatorname{Jump}} %
\newcommand{\Gl}{\operatorname{Gl}} %
\newcommand{\Fil}{\operatorname{Fil}} %
\newcommand{\ALSpec}{\operatorname{Spec_L}} 
 
\begin{document}

\newcommand{\spacing}[1]{\renewcommand{\baselinestretch}{#1}\large\normalsize}
\spacing{1.14}

\title{Isospectral surfaces with distinct covering spectra via Cayley Graphs}

\author[B. de Smit]{Bart de Smit $^\ast$}
\address{Universiteit Leiden\\Mathematisch Instituut\\Postbus 9512 \\2300 RA
Leiden} \email{desmit@math.leidenuniv.nl}
\thanks{$^\ast$ Funded in part by the European Commission under contract
MRTN-CT-2006-035495}

\author[R. Gornet]{Ruth Gornet $^\dagger$}
\address{University of Texas \\ Department of Mathematics \\ Arlington, TX
76019 } \email{rgornet@uta.edu}
\thanks{$^\dagger$ Research partially supported by NSF grant DMS-0204648}

\author[C. J. Sutton]{Craig J. Sutton$^\sharp$}
\address{Dartmouth College\\ Department of Mathematics \\ Hanover, NH 03755}
\email{craig.j.sutton@dartmouth.edu}
\thanks{$^\sharp$ Research partially supported by NSF grant DMS-0605247 and a 
Career Enhancement Fellowship from the Woodrow Wilson National Fellowship Foundation}




\begin{abstract}
The covering spectrum is a geometric invariant of a Riemannian manifold, 
more generally of a metric space, that measures the size of its one-dimensional 
holes by isolating a portion of the length spectrum.
In a previous paper we demonstrated that the covering spectrum 
is not a spectral invariant of a manifold in dimensions three and higher.
In this article we give an example of two isospectral Cayley graphs that 
admit length space structures with distinct covering spectra.
From this we deduce the existence of infinitely many pairs of 
Sunada-isospectral surfaces with unequal covering spectra.
\end{abstract}

\maketitle



\section{Introduction}

The covering spectrum of a complete length space $X$ is a geometric invariant
introduced by Sormani and Wei \cite{SW} 
that measures the size of one dimensional holes in 
$X$ by isolating the lengths of certain closed geodesics that 
are minimal in their free homotopy class. 
More specifically, the covering spectrum can be computed by 
considering a certain family $\{\tX^{\delta}\}_{\delta >0}$ of regular coverings of 
$X$, where $\tX^{\delta}$ covers $\tX^{\delta'}$ for $\delta' > \delta$. 
The covering spectrum of $X$ then consists of those values of $\delta$ where
the isomorphism type of the cover changes.
That is, the covering spectrum is the collection of $\delta$'s 
where there is a ``jump'' in the step function $\delta \mapsto \tX^{\delta}$. 
One can see that a ``jump'' at time $\delta$ corresponds 
to the appearance of a ``hole'' of circumference $2 \delta$ or, in the case of manifolds, a 
non-trivial free homotopy class whose shortest geodesic is of length $2\delta$ 
(cf. Proposition~\ref{Prop:NewGeo}).

As an example, we follow \cite[Example 2.5]{SW} and consider the flat $2 \times 3$ torus $X= S^1(2) \times S^1(3)$, 
where $S^1(c)$ denotes the circle of circumference $c$. 
One can see through the covering spectrum algorithm 
(see \ref{CSA}) that the $\delta$-covers of $X$ are given by 
$\tX^\delta = \R^2$ (for $0 < \delta \leq 1$),
$\tX^\delta = S^1(2) \times \R$ (for $1< \delta \leq \frac{3}{2}$ ) 
and $\tX^\delta = X$ (for $\delta > \frac{3}{2}$).
Hence, the covering spectrum of $X$ is given by 
$\{1, \frac{3}{2} \}$.
Each of these numbers is the circumference of a ``hole'' in this torus and 
also half the length of the minimal closed geodesics
in the free homotopy classes associated to the standard generators 
$(1,0)$ and $(0,1)$ in $\pi_1(X) = \Z \times \Z$.  

As there is a long-standing interest in the relationship 
between the Laplace spectrum of a manifold and its length spectrum
(i.e., the collection of lengths of closed geodesics), 
it is natural to wonder about the mutual influences between the Laplace spectrum 
and the covering spectrum. Specifically, one would like to know if the 
covering spectrum is a spectral invariant. 
In \cite[Ex.~10.5]{SW}, Sormani and Wei demonstrated that certain isospectral manifolds 
arising through Sunada's method \cite{Sunada} share the same covering spectrum, suggesting that 
the covering spectrum might at least be an invariant of this technique.
However, through an analysis of the group theory and geometry underlying the covering spectrum we 
constructed Sunada-isospectral manifolds with distinct covering spectra 
in dimensions three and higher \cite[Sec. 8]{DGS}, thus showing that 
the covering spectrum is not a spectral invariant. 

Indeed, we recall that a triple of finite groups $(G, H_1, H_2)$, where
$H_1$ and $H_2$ are subgroups of $G$, is said to be a \emph{Gassmann-Sunada} triple 
if for any conjugacy class $C \subset G$ the sets $C \cap H_1$ 
and $C \cap H_2$ are of the same order.
If $(G, H_1, H_2)$ is a Gassmann-Sunada triple where $G$ acts 
freely and isometrically on a Riemannian manifold $M$, 
then Sunada's theorem (Theorem~\ref{Thm:SunadaPesce})  
tells us that the quotient Riemannian manifolds $H_1 \bs M$ and $H_2\bs M$ are isospectral,
and the isospectral pairs arising in this manner are said to be \emph{Sunada-isospectral}.
Sunada's theorem provided the first systematic method for constructing isospectral manifolds 
and---along with its numerous variants---accounts for many of the 
isospectral pairs in the literature (see \cite{Gordon} for a 
retrospective on Sunada's method).
In \cite{DGS} we discovered an analogous algebraic condition and theorem concerning 
the covering spectrum. Specifically, we showed the following.

\begin{thm}[\cite{DGS} Theorem 2.2]\label{Thm:DGS}
Let $(G, H_1, H_2)$ be a triple of finite groups, where 
$H_1$ and $H_2$ are subgroups of $G$, and suppose $G$ acts freely 
on a closed manifold $M$. If $M$ is 
simply-connected and of dimension at least $3$, then 
the following are equivalent:
\begin{enumerate}
\item $H_1$ and $H_2$ are \emph{jump equivalent} 
subgroups of $G$;
that is, for any subsets $S, T \subset G$ that 
are stable under conjugation we have
$$\langle H_1 \cap S \rangle = \langle H_1 \cap T \rangle \Longleftrightarrow 
\langle H_2 \cap S \rangle = \langle H_2 \cap T \rangle. $$
In which case we 
say that $(G, H_1, H_2)$ is a \emph{jump triple}.
\item For every $G$-invariant Riemannian metric on $M$, 
the quotient Riemannian manifolds $H_1 \bs M$ 
and $H_2 \bs M$ share the same covering spectrum.
\end{enumerate}
\end{thm}

\noindent
We then went on to produce examples of Gassmann-Sunada triples that 
are not jump triples. As a consequence we were able to construct 
Sunada-isospectral manifolds 
with distinct covering spectra in all dimensions greater than or equal to three.
However, our method for building such examples appears to fail in dimension two 
as the arguments underlying Theorem~\ref{Thm:DGS} depend in part on the fact that 
all free homotopy classes in dimension 
three and higher can be represented by simple closed curves.

In this article we take a different approach in order to establish the 
following theorem showing that 
the covering spectrum is not a spectral invariant of a surface.

\begin{thm}\label{Thm:Main}
For each integer $n\ge1$ there exist pairs of closed connected
Sunada-isospectral surfaces of genus $7n+1$ with distinct covering spectra.
\end{thm}

To prove the result above we construct certain Cayley graphs 
associated to a well-known Gassmann-Sunada 
triple \cite[Ex.~11.4.7.]{Buser},
and endow them with suitably chosen edge lengths
so that the resulting length spaces $X_1$ and $X_2$ 
have distinct covering spectra. 
We then build continuous families of Riemannian 
manifolds $\{(M_1, g_1^\epsilon)\}_{\epsilon > 0}$ 
and $\{(M_2, g_2^\epsilon)\}_{\epsilon > 0}$ that are pairwise 
Sunada-isospectral for each $\epsilon > 0$ and 
such that for each $i = 1,2$ the family $(M_i, g_i^\epsilon)$ 
converges with respect to the Gromov-Hausdorff metric to $X_i$ as $\epsilon$ tends to zero.
Then, since the covering spectrum is continuous under 
Gromov-Hausdorff convergence \cite[Corollary 8.5]{SW},
it follows that for $\epsilon$ sufficiently small $(M_1, g_1^\epsilon)$ 
and $(M_2, g_2^\epsilon)$ are Sunada-isospectral manifolds with distinct 
covering spectra.

Currently, it is not clear to us whether any of the aforementioned pairs 
consist of Riemann surfaces. 
It would appear to be an interesting problem to 
determine whether the covering spectrum is a spectral invariant 
among such spaces.

The outline of this article is as follows. In Section~\ref{Sec:Covering} we 
will formulate the precise definition of the covering spectrum, 
present a useful algorithm for computing it, and discuss its 
behavior under Gromov-Hausdorff convergence. 
In Section~\ref{Sec:Cayley} we recall the notion of a Cayley graph, we
produce the length spaces $X_1$ and $X_2$ mentioned above, and show their covering spectra are distinct.
In Section \ref{Sec:Surfaces} we review how the Cayley graphs relate
to the isospectral problem, and we finish the proof of Theorem \ref{Thm:Main}
as outlined.


\section{The Covering Spectrum}\label{Sec:Covering}

In this section we will recall the definition of the covering spectrum as given in \cite{SW}.
Since the primary focus of this paper will be the covering spectra of compact
semi-locally simply-connected length spaces we will only formulate our definition 
in this setting. However, we note that in \cite{DGS} we 
extended this definition of the covering spectrum to arbitrary metric spaces. 

We will begin by recalling the notion of a filtration of a group.

\begin{dfn}
Let $G$ be a group. A \emph{filtration} $\Fil^\bullet G$ of $G$ is 
a family $\{\Fil^{i} G\}_{i \in I}$ of subgroups of $G$, where $i$ ranges 
over an ordered index set $I$, such that $\Fil^i G \subset \Fil^j G$ whenever $i < j$.
The \emph{jump set} of the filtration is the subset of $I$ given by
$$\Jump ( \Fil^\bullet G) = \{ i \in I : \Fil^i G \neq \Fil^j G \mbox{ for all } j > i \}.$$
\end{dfn}

\begin{exa}[Filtrations Induced by Class Functions]\label{Exa:Filtration}
Let $G$ be a group and $m: G \to \nnR$ a class function; i.e., 
$m$ is constant on the conjugacy classes of $G$. Then $m$ induces 
a filtration $\Fil^\bullet_m G = \{\Fil^{\delta}_{m} G \}_{\delta > 0}$, where
$$\Fil^\delta_m G := \langle g \in G : m(g) < \delta \rangle.$$
Such filtrations will prove to be useful in our discussion of the covering spectrum.
\end{exa}

We now observe that given a connected and compact semi-locally 
simply-connected length space $(X,d)$ its metric structure induces 
a natural filtration (indexed by the positive reals) on 
its fundamental group $\pi_1(X, x_0)$. Indeed, for each $\delta > 0$ 
let $\mU_{\delta}$ be the open cover of $X$ consisting of all the 
open $\delta$-balls in $(X,d)$. Then let $\Fil^\delta \pi_1(X, x_0)$ 
be the subgroup of $\pi_1(X, x_0)$ generated by classes 
of the form $[\alpha * \beta * \bar{\alpha}]$, where $\beta : [0, 1] \to X$ 
is a loop that is completely contained in some open $\delta$-ball, 
$\alpha$ is a path from $x_0$ to $\beta(0)$ and $\bar{\alpha}(t) = \alpha(1-t)$ 
is its reverse. 
The subgroup $\Fil^\delta \pi_1(X, x_0)$  is normal in $\pi_1(X, x_0)$ and 
the conjugacy classes of the generators $[\alpha * \beta * \bar{\alpha}]$ 
of $\Fil^\delta \pi_1(X, x_0)$ represent the free homotopy classes 
of loops that have representatives completely contained in a $\delta$-ball. In this fashion we obtain 
a filtration $\Fil^\bullet \pi_1(X, x_0) = \{ \Fil^\delta \pi_1(X, x_0) \}_{\delta > 0}$ 
of $\pi_1(X, x_0)$.

\begin{dfn}
Let $(X,d)$ be a connected and compact semi-locally simply-connected length space. 
We then define the \emph{covering spectrum} of $(X,d)$ to be the set 
$$\CovSpec(X,d) = \Jump (\Fil^\bullet \pi_1(X, x_0)).$$
\end{dfn}

Now, since each subgroup $\Fil^\delta \pi_1(X, x_0)$ of 
$\pi_1(X, x_0)$ is normal, the Galois theory of covering spaces \cite[Theorem 2.5.13]{Spanier} tells us 
that we may associate to each subgroup a connected regular covering 
space $\tX^\delta \to X$ known as the \emph{$\delta$-cover} of $X$. 
The $\delta$-cover has the property that if $Y \to X$ is any covering space 
that evenly covers every $\delta$-ball of $X$, then $\tX^\delta$ is a covering space of $Y$ (cf. \cite[Lemma 2.5.11]{Spanier}).
Hence, we obtain a tower $\{\tX^\delta \}_{\delta > 0}$ of regular covers of $X$ 
and the covering spectrum detects when we have a change in the isomorphism 
type of the cover. That is, $\delta \in \CovSpec(X,d)$ if and only if for 
all $\epsilon > \delta$, $\tX^\epsilon$ and $\tX^\delta$ are non-isomorphic as
covering spaces. In the event that we are considering a 
compact Riemannian manifold $(M,g)$ we see that the covering spectrum 
also contains information about the length spectrum.

\begin{prop}[\cite{DGS} Corollary 4.5]\label{Prop:NewGeo}
Let $(M,g)$ be a closed Riemannian manifold. Then 
\begin{enumerate}
\item $2 \CovSpec(M,g)$ is a subset of the length spectrum.
\item if $\delta \in \CovSpec(M,g)$, then $2\delta$ is the 
length of the shortest (smoothly) closed geodesic in $(M,g)$ that has a 
lift to $\tM^{\delta}$ that is not a closed loop. 
\end{enumerate}
\end{prop}

In order to compute the length spectrum it will be useful 
to introduce the minimum marked length map.

\begin{dfn}
Let $(X,d)$ be a connected and compact semi-locally 
simply-connected length space with fundamental group $\pi_1(X)$ and
simply-connected universal cover $(\tX, \tilde{d})$. 
Then the \emph{minimum marked length map} 
$m: \pi_1(X) \to \nnR$ is given by 
$$m([\gamma]) \equiv \inf_{\tilde{x} \in \tX} \tilde{d}(\tilde{x}, \gamma \cdot \tilde{x}),$$
where $\pi_1(X)$ acts on $\tX$ via deck transformations.
\end{dfn}

\noindent
The minimum marked length map $m: \pi_1(X) \to \nnR$ assigns 
to each $[\gamma] \in \pi_1(X)$
the length of the shortest closed geodesic in the free 
homotopy class determined by $[\gamma]$.
Since free homotopy classes correspond to conjugacy classes 
in $\pi_1(X)$ we see that $m$ is a class function.
Hence, by Example~\ref{Exa:Filtration} $m$ induces the 
filtration $\Fil^\bullet_m \pi_1(X)$ and we have the following.

\begin{prop}[cf. \cite{DGS} Proposition 2.5]\label{prop:filter}
Let $(X,d)$ be a connected and compact semi-locally 
simply-connected length space with minimum marked length 
map $m : \pi_1(X) \to \nnR$. Then 
$$\CovSpec(X,d) = \frac{1}{2} \Jump( \Fil^\bullet_m \pi_1(X)).$$ 
\end{prop}

The above leads to the 
following algorithm for computing the covering spectrum of 
a compact semi-locally simply-connected length space as found in \cite[p. 54]{SW} and 
\cite[4.6]{DGS}.

\begin{csa}\label{CSA}
Let $(X,d)$ be a compact semi-locally 
simply-connected length space with minimum marked length map
$m: \pi_1(X) \to \nnR$. Then let
$$
\begin{array}{l}
\delta_1=\inf\{m(\sigma)\colon \sigma \in \pi_1(X), \sigma \ne 1\} > 0 \\
\delta_2=\inf\{m(\sigma)\colon \sigma \in \pi_1(X), \sigma \not\in \langle  \tau \in \pi_1(X) \colon\;
m(\tau)\le\delta_1\rangle \} > \delta_1\\ 
\delta_3=\inf\{m(\sigma)\colon \sigma \in \pi_1(X), \sigma \not\in \langle  \tau \in \pi_1(X) \colon\;
m(\tau)\le\delta_2\rangle \} > \delta_2\\ 
\cdots
\end{array}
$$
where we continue until $\langle  \sigma \in \pi_1(X) \colon\; m(\sigma) \le \delta_r \rangle = \pi_1(X)$.
Now since $X$ is a compact length space we know by \cite[Lemma 4.6]{SW} 
that the image of $m$ is closed and discrete. It then follows from the fact that $\pi_1(X)$ 
is finitely generated that the above process stops in finitely many steps at some $\delta_r$, 
and by the previous proposition we see
$$\CovSpec(X,d) = \{ \frac{1}{2} \delta_1 , \ldots , \frac{1}{2}\delta_r \}.$$
\end{csa}

\begin{rem}
For an application of the covering spectrum algorithm to the case of flat tori the reader is encouraged to see \cite[Example 4.7]{DGS}.
In particular, this example shows that for a flat $n$-torus, the covering spectrum need not be equal to the 
successive minima of the corresponding lattice $\mL$ in $\R^n$.
\end{rem}

We conclude this section by recalling that in Proposition~\ref{Prop:NewGeo}
we observed that for any compact length space $X$ we have 
$2\CovSpec(X) \subset \ALSpec(X)$, 
where $\ALSpec(X)$ denotes the length spectrum of $X$.
Now, it is well known that the length spectrum does not 
behave well under Gromov-Hausdorff 
convergence: lengths may disappear or 
suddenly appear in the limit (cf. \cite[Examples 8.1-8.3]{SW}).
However, in Section~\ref{Sec:Cayley} we will use the fact that, in contrast, 
the covering spectrum is continuous under Gromov-Hausdorff convergence.

\begin{prop}[\cite{SW} Corollary 8.5]\label{Prop:GHConv}
If $X_i$ is a sequence of compact length spaces converging to a compact length 
space $Y$ in the Gromov-Hausdorff sense, then the covering spectra converge in the 
Hausdorff sense as subsets of $\R$. That is, if $X_i \stackrel{\rm{GH}}{\longrightarrow}Y$, then 
$$\lim_{i \to \infty} d_{H}( \CovSpec(X_i) \cup \{0\}, \CovSpec(Y)\cup\{0\} ) = 0,$$
where $d_H$ denotes the Hausdorff distance between subsets of $\R$.
\end{prop} 

\noindent
Hence, twice the covering spectrum determines a part of the length 
spectrum of a compact length space that behaves well 
under Gromov-Hausdorff convergence.

\begin{rem}
We refer the reader to Sections 7 and 8 of \cite{SW} for 
the relevant details on Gromov-Hausdorff convergence and the proof of the 
proposition above. 
\end{rem}

\section{Cayley Graphs associated to the Fano plane}\label{Sec:Cayley}
\noindent

In this section we will introduce certain Cayley graphs and associated length
spaces. In the next section we will use these examples to construct
isospectral surfaces with distinct covering spectra.

By a \emph{graph} we will mean a quadruple $(V,E, o,t)$ where $V$ and $E$ are
the sets of vertices and edges (respectively), and $o$ and $t$ are maps from $E$ to $V$
associating to each edge its origin and target.  
In other words, we consider directed graphs with possible self-edges and
multiple edges between the same vertices.  
A morphism $f$ from a graph $(V,E,o,t)$ to a graph $(V',E',o',t')$ is a pair of
maps $f_v\colon\;V\to V'$ and $f_e\colon\;E\to E'$ for which $o'\circ f_e=
f_v\circ o$ and $t'\circ f_e= f_v\circ t$.

If $G$ is a group acting \emph{on the right} on a set $V$, and $S\subset G$ is
a subset, then we can consider the (\emph{generalized}) \emph{Cayley graph} $V[S]=(V,E,o,t)$ associated to $V$ and $S$, where
$E=V\times S$ with $o(v,s)=v$ and $t(v,s)=vs$.  Note that we can label the
edges by $S$ according to their second coordinate---this will be referred to 
as the \emph{color} or \emph{type} of the edge---and that for every
$s\in S$ every vertex has a single incoming and outgoing edge of type $s$.

If for $V$ we take a finitely generated group $G$ with the $G$-action by right multiplication
and we let $S$ be a set of generators of $G$,
then we obtain the usual Cayley graph $G[S]$ of the group $G$ (cf. \cite[Chp. 4]{Harpe}).  The group $G$ then has a left-action
by graph automorphisms on $G[S]$ where $G$ acts on the vertices by left
multiplication. The action preserves the labeling of the edges by $S$.
For any subgroup $H$ of $G$ the quotient graph $H\backslash
(G[S])$ is then naturally isomorphic to the generalized Cayley graph $(H\backslash G)[S]$.
The graph $(H\backslash G)[S]$ is also known as the \emph{Schreier graph} associated to $H \backslash G$ and $S$ (cf. \cite[Chp. 4]{Harpe}).

In our application, the group $G$ will be the group $G=\Gl_3(\ff_2)$, which we
will consider with its right action on the Fano plane, the projective plane
over the field $\ff_2$ of $2$ elements.
The set $V_1$ of points of the Fano plane is labeled by non-zero
row vectors of length $3$ over $\ff_2$, which are multiplied on the right by
matrices in $G$. 
A second set $V_2$ with right $G$-action that we will consider is the
set of lines in the Fano plane. We can label $V_2$ by the same set
of non-zero vectors in $\ff_2^3$
where the line through two distinct points $v_1,v_2\in V_1$ is labeled by the
unique nonzero vector $w\in\ff_2^3$ so that $\langle v_1, w\rangle =\langle
v_2,w\rangle=0$ under the standard inner product 
$\langle\cdot,\cdot\rangle$ on $\ff_2^3$; see
Figure \ref{Fig:Fano}. 
By computing the stabilizer in $G$ of the element represented by $100$ 
in each of $V_1$ and $V_2$ we 
see that we have $G$-set isomorphisms $V_1\cong H_1\backslash G$, and $V_2\cong H_2\backslash G$, where
$$
H_1 =
\left(\begin{array}{ccc}1 & 0 & 0 \\ * & * & * \\ * & * & *\end{array}\right) 
\mbox{ and }
H_2 =
\left(\begin{array}{ccc} 1 & * & * \\0 & * & * \\0 & * & *\end{array}\right).
$$ 

Let us now consider the subset $S = \{A, B\} \subset G=\Gl_3(\ff_2)$ consisting of the two matrices
$$
A= \left( \begin{array}{rrr} 1&1&0\\ 0&0&1\\ 0&1&0 \end{array} \right), \qquad
B= \left( \begin{array}{rrr} 0&1&0\\ 0&0&1\\ 1&0&0 \end{array} \right).$$
The action of $A$ on the Fano plane is given by dotted arrows in Figure
\ref{Fig:Fano}, while the action of $B$ is rotation by 120 degrees
counterclockwise.  By rearranging the vertices we see that the
Cayley graph $V_1[A,B]$ is the graph in Figure \ref{Fig:Graph1}
where the dotted arrows indicate edges associated to $A$ and
the solid arrows indicate edges associated to $B$.
By tracking the action of $A$ and $B$ on the lines of the Fano plane
(or on the three points each line contains) we construct the Cayley graph 
$V_2[A,B]$ as in Figure~\ref{Fig:Graph2}.
These graphs are also given in \cite[Ex.~11.4.7]{Buser} and \cite{Brooks}.

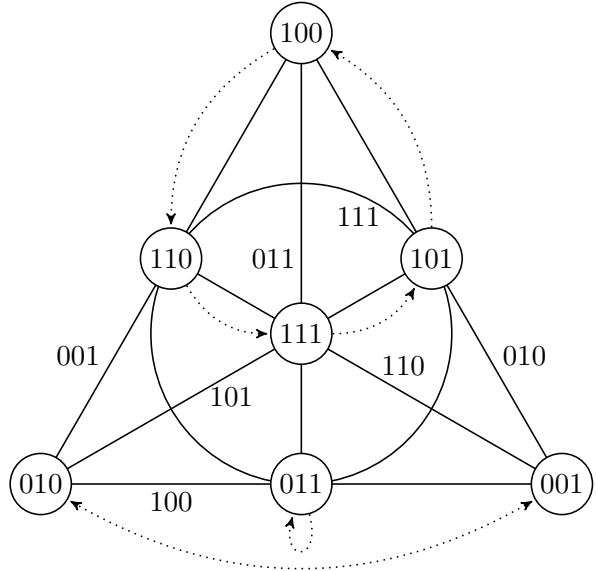
\begin{figure}
\begin{tikzpicture}[inner sep=2pt, shorten >= 1pt, >=stealth', semithick]
\node [below left,inner sep=0pt] at (58:2) {$111$};
\draw (0,0) circle (2cm);
\draw (90:4) -- (210:4) node[pos=0.75,above left]{$001$};
\draw (90:4) -- (330:4) node[pos=0.75,above right]{$010$};
\draw (90:4) -- (90:-2) node[pos=0.5,left]{$011$};
\draw (330:4) -- (210:4) node[pos=0.75,below, inner sep=3pt]{$100$};
\draw (210:4) -- (210:-2) node[pos=0.43,below right, inner sep=0pt]{$101$};
\draw (330:4) -- (330:-2) node[pos=0.47,above right, inner sep=1pt]{$110$};
\node[circle,draw,fill=white] (z100) at (90:4) {$100$};
\node[circle,draw,fill=white] (z010) at (210:4) {$010$};
\node[circle,draw,fill=white] (z001) at (330:4) {$001$};
\node[circle,draw,fill=white] (z101) at (30:2) {$101$};
\node[circle,draw,fill=white] (z110) at (150:2) {$110$};
\node[circle,draw,fill=white] (z011) at (270:2) {$011$};
\node[circle,draw,fill=white] (z111) at (0,0) {$111$};
\path[->] 
    (z100) edge[dotted, bend right] node{} (z110)
    (z110) edge[dotted, bend right] node{} (z111) 
    (z111) edge[dotted, bend right] node{} (z101) 
    (z101) edge[dotted, bend right] node{} (z100)
    (z011) edge[dotted, loop below] node{} (z011);
\path[<->,shorten <= 1pt] 
    (z010) edge[dotted, bend right] node{} (z001);
\end{tikzpicture}
\caption{The Fano plane with seven points and seven lines,
with the action of the matrix $A$ given by dotted arrows.}
\label{Fig:Fano}
\end{figure}

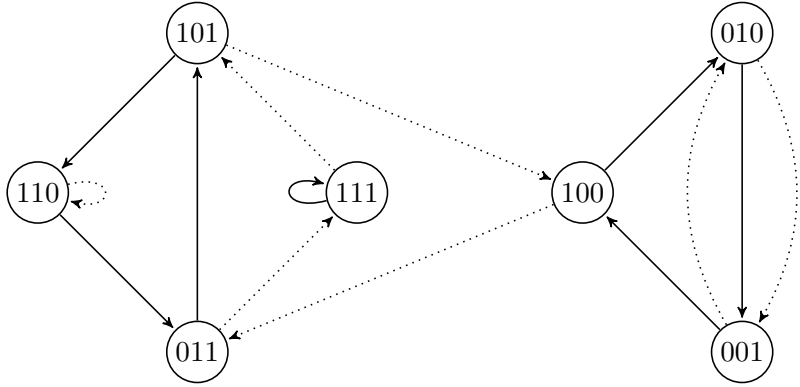
\begin{figure}
\begin{tikzpicture}[->,>=stealth',shorten >=1pt,auto,node distance=3cm,
                    semithick,inner sep=2pt]
  \node[circle,draw,fill=white] 	       (4)                              {$110$};
  \node[circle,draw,fill=white]         (1) [above right of= 4] 		{$101$};
  \node[circle,draw,fill=white]         (3) [below right of= 4] 		{$011$};
  \node[circle,draw,fill=white]         (0) [above right of= 3] 		{$111$};
  \node[circle,draw,fill=white]         (6) [right of= 0]       		{$100$};
  \node[circle,draw,fill=white]         (2) [above right of= 6]       	{$010$};
  \node[circle,draw,fill=white]         (5) [below right of= 6]          {$001$};

  \path[->] (4) edge [loop right,dotted] node {}   (4)
            	     edge []  node {}  (3)
                (1) edge [] 	node {}  (4)
                     edge [dotted]     node {} (6)
                (3) edge [dotted]  node {} (0)
                      edge node {} (1)
                (0) edge [loop left] node {} (0)
                     edge [dotted] node {} (1)
                (6) edge  node{} (2)
                     edge [dotted]  node{} (3)
                (2) edge node {} (5)
                      edge [bend left,dotted] node {} (5)
                (5) edge [bend left,dotted] node{} (2)
                      edge node{} (6);

\end{tikzpicture}
\caption{The Cayley graph $V_1[A,B] =H_1\backslash G[A,B]$ of $A$ and $B$ acting on the points in the Fano plane.}
\label{Fig:Graph1}
\end{figure}

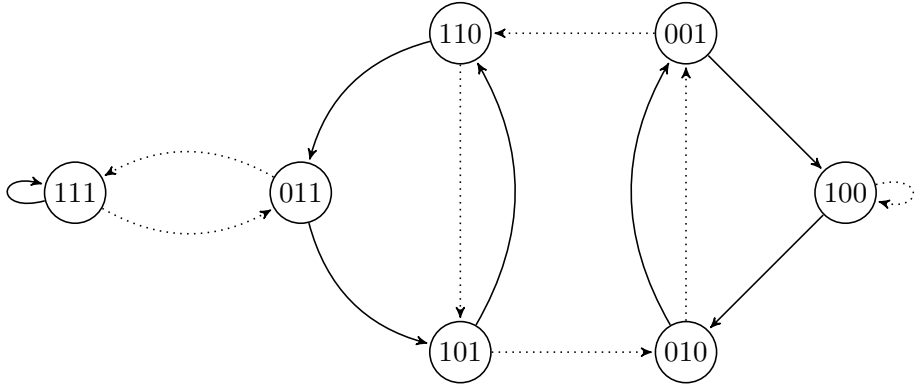
\begin{figure}
\begin{tikzpicture}[->,>=stealth',shorten >=1pt,auto,node distance=3cm,
                    semithick,inner sep=2pt]
  \tikzstyle{every state}=[
                                inner sep=2pt,circle,draw,fill=white]

  \node[circle,draw,fill=white] 	       (3)             			{$111$};
  \node[circle,draw,fill=white]         (0) [right of= 3] 		{$011$};
  \node[circle,draw,fill=white]         (4) [above right of= 0] 		{$110$};
  \node[circle,draw,fill=white]         (6) [below right of= 0] 		{$101$};
  \node[circle,draw,fill=white]         (5) [right of= 4]       		{$001$};
  \node[circle,draw,fill=white]         (2) [right of= 6]       		{$010$};
  \node[circle,draw,fill=white]         (1) [above right of= 2]      	{$100$};

  \path[->]     (0) edge [dotted, bend right] node {}   (3)
            	     edge [bend right]  node {}  (6)
                (3) edge [loop left] 	node {}  (0)
                     edge  [bend right,dotted]    node {} (0)
                (1) edge [loop right,dotted]   node {} (1)
                      edge node {} (2)
                (2) edge [bend left] node {} (5)
                     edge [dotted] node {} (5)
                (4) edge [dotted]  node{} (6)
                     edge [bend right] node{} (0)
                (5) edge [dotted] node {} (4)
                      edge node {} (1)
                (6) edge [dotted] node{} (2)
                      edge [bend right] node{} (4);
\end{tikzpicture}
\caption{The Cayley graph $V_2[A,B]=H_2\backslash G[A,B]$ of $A$ and $B$ acting on the lines in the Fano plane.}
\label{Fig:Graph2}
\end{figure}

The \emph{geometric realization} of a graph $(V,E,o,t)$ is the topological
space that one obtains by glueing closed intervals onto the
set $V$ as follows: for each $e\in E$ one adds an interval
$[0,1]$ identifying $0$ with $o(e)\in V$ and $1$ with $t(e)\in V$.
If one designates a positive real number $l_e$ as the
length of an edge $e$ then one can give the interval assocated to $e$ 
length $l_e$ and in this way one obtains a length space.

For our Cayley graph $G[A,B]$ we will assign a length $l_A>0$ to all edges
associated with $A$ and $l_B>0$ to all edges associated with $B$, and we
thus obtain a length space $X(l_A,l_B)$. The group $G$ acts by isometries on
$X(l_A,l_B)$ and it acts freely, i.e., no non-trivial element of $G$ has a
fixed point on $X(l_A,l_B)$. The quotient length space $H_1\backslash
X(l_A,l_B)$ is the length space $X_1$ associated to $V_1[A,B]$ and  $H_2\backslash
X(l_A,l_B)$ is the length space $X_2$ associated to $V_2[A,B]$.

\begin{prop}\label{Prop:distinct}
If $0<l_A <l_B < \frac{3}{2} l_A$, then $l_A+ \frac{1}{2}l_B$ lies in the covering spectrum of $X_1=H_1\backslash X(l_A,l_B)$ but not in 
the covering spectrum of $X_2=H_2\backslash X(l_A,l_B)$.
\end{prop}

\begin{proof}
We see from the assumption 
that 
$$l_A < l_B < 2l_A < l_A + l_B < 2l_{B} < 3l_{A} < 2l_{A} + l_{B}$$
represent the lengths of the non-trivial minimal loops of length 
below $l_{A} + 2l_{B}$. These lengths form
the first column of the table in Figure \ref{Fig:table}.

The second and third column in the table lists closed loops in
$X_1$ and $X_2$ of the given lengths.
In this table $A_{001}$ denotes the path along the
edge associated to $A$ leaving from the vertex with label $001$.
One has to inspect the appropriate graph to see what the endpoint
of $A_{001}$ is: in $X_1$ it is $010$ and in $X_2$ it is $110$.
Paths are composed according to the following convention: if the endpoint of a path $p$ is the starting
point of a path $q$ then $p*q$ runs through $p$ at double speed on the interval
$[0,1/2]$ and then through $q$ on $[1/2,1]$. 

By inspecting the graphs in Figures \ref{Fig:Graph1} and \ref{Fig:Graph2} one 
deduces that every closed loop in $X_1$ or $X_2$ that is of length
below $l_A+2l_B$, is freely homotopic to either a
constant loop, or to a loop in the table, or to the inverse of a loop in the
table in Figure \ref{Fig:table}. 
Note also that the lengths of the loops in the table are minimal within
their free homotopy class.

\begin{figure}
\begin{tabular}{| l || c | c |} \hline
length   &  $X_1=H_1\backslash X(l_A,l_B)$  &  $X_2=H_2\backslash X(l_A,l_B)$ \\ \hline \hline
$l_A$ &  $\sigma_1 = A_{110}$ & $\tau_1 = A_{100}$ \\ \hline
$l_B$ & $\sigma_2 = B_{111}$ & $\tau_2 = B_{111}$ \\ \hline 
$2l_A$ & $\sigma_1^2$ & $ \tau_1^2$ \\
& $\sigma_3 = A_{010}*A_{001}$ & $ \tau_3 = A_{111}*A_{011}$ \\ \hline
$l_A + l_B$ 
&  $\sigma_4 = B_{010}*A_{001}$& $\tau_4 = A_{010}*B_{010}^{-1}$ \\ 
&  $\sigma_5 = B_{010}*A_{010}^{-1}$& $\tau_5 = A_{110}*B_{101}$ \\ 
\hline 
$2 l_B$ & $\sigma_2^2$  &  $\tau_2^2$\\ \hline 
$3 l_A $ & $\sigma_1^3$  & $\tau_1^3$ \\ \hline 
$2 l_A + l_B$ &
$\sigma_6 = B_{011}*A_{101}*A_{100}$ &
  $\tau_2*\tau_3$ \\ 
&$\sigma_{7} = B_{011}*A_{111}^{-1}*A_{011}^{-1}$&
  $\tau_2*\tau_3^{-1}$ \\ 
\hline
\end{tabular}
\caption{The minimal loops in $X_1$ and $X_2$ of length below $l_A+2l_B$}
\label{Fig:ShortGeodesics}
\label{Fig:table}
\end{figure}

Denoting the minimum marked length map for $X_2$ by $m$
one sees that $\Fil_m^{2l_A+l_B} (\pi_1(X_2))$ is the normal
subgroup of $\pi_1(X_2)$ generated by all homotopy classes of loops that are
freely homotopic to one of $\tau_1, \ldots ,\tau_5$, which is also
equal to $\Fil_m^{2l_A+l_B+\epsilon} (\pi_1(X_2))$ for any $0<\epsilon\le l_B - l_A$.  
We deduce that $2l_A+l_B$ is not a jump of the filtration $\Fil_m^\bullet\pi_1(X_2)$,
so that by Proposition \ref{prop:filter} the number $l_A+ \frac{1}{2}l_B$ is not in the
covering spectrum of $X_2$.

To complete the proof we will show that the class $c$ in $\pi_1(X_1,x_0)$ 
of any loop in $X_1$ based at $x_0$, that is freely homotopic
to $\sigma_6$, does not lie in the normal subgroup $N$ of
$\pi_1(X_1,x_0)$ generated by the classes of all loops based at $x_0$
that are freely homotopic to one of $\sigma_1,\ldots,\sigma_5$.
To see this, consider the topological quotient space $K$ of $X_1$
that one obtains by contracting the loops $\sigma_1,\ldots,\sigma_5$ to a
point. Then $N$ lies in the kernel of the induced group homomorphism
$\pi_1(X_1)\to \pi_1(K)$, but one sees from the graph in Figure
\ref{Fig:Graph1} that $c$ does not map to the trivial class.  Since
$\sigma_6$ has length $2l_A+l_B$, and this length is minimal in its free
homotopy class, it follows with Proposition \ref{prop:filter} that $l_A+ \frac{1}{2}l_B$ lies in the covering spectrum of
$X_1$.  
\end{proof}

\section{From Cayley graphs to isospectral surfaces}\label{Sec:Surfaces}

Suppose that $G$ is a group and $S$ is a subset of $G$.
For any set $V$ with a right action of $G$ we defined the Cayley graph
$V[S]$ with vertex set $V$ and edge set $E$ consisting of an edge from 
$x$ to $xs$ for each $s\in S$ and $x \in V$.
We can construct a surface out of this graph as follows (cf. \cite[Chapter 11]{Buser}).

Suppose that we have a compact connected surface of genus $g$.
If $n=\#S$, then we add $n$ handles to the surface,
and on this surface $M$ we choose a Riemannian metric.
We now cut open the handles to obtain a compact surface $M'$
with $2n$ border components $\{B_s, \overline B_s\colon\; s\in S\}$,
where for each $s\in B$ a homeomorphism $B_s\to \overline B_s$
is given so that identifying $B_s$ with $\overline B_s$ for all $s\in S$
gives us $M$.  In Figure \ref{Fig:Genus2} this is illustrated for $g=0$ and
$n=2$.

\begin{figure}
\includegraphics[width = 2.5in]{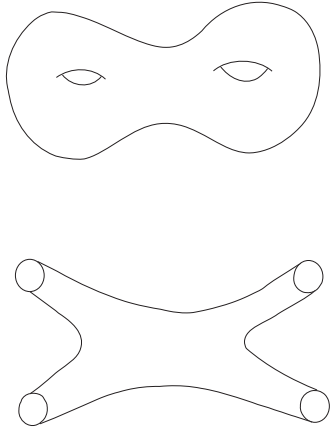}
\caption{Cutting a surface of genus two along two curves}\label{Fig:Genus2}
\end{figure}

With this building block we can now make a surface as follows.  Take the
surface $M'\times V$ and for every edge from an element $x\in V$ to $xs\in V$
of color $s\in S$ we identify $B_s\times\{x\}$ with $\overline
B_s\times \{xs\}$ using the homeomorphism $B_s\to \overline B_s$.  We will
denote this quotient by $M[V;S]$. It is a compact surface,
and it is not hard to see that if $V[S]$ is connected, then 
the genus of $M[V;S]$ is $1+(g+n-1)\#V$.

A Riemannian metric on $M$ then gives rise to a Riemannian metric on
$M[V;S]$ and it follows that any color-preserving graph automorphism of $V[S]$ induces an isometry
of $M[V;S]$. Thus, for any subgroup $H$ of $G$ we find that
$H$ acts freely and by isometries on $M[G;S]$ and the manifolds
$H\backslash M[G;S]$ and $M[H\backslash G;S]$ are isometric.

This set-up allows us to produce isospectral manifolds
by Sunada's method.
We recall that a triple of groups $(G, H_1, H_2)$ where $H_1, H_2 \leq G$ is
said to be a \emph{Gassmann-Sunada triple} if for any conjugacy class $C
\subset G$ we have $\#(C \cap H_1) = \#(C \cap H_2)$.  
This group theoretic notion identifies the configurations of isospectral
manifolds in the sense of the following theorem.

\begin{thm}[\cite{Sunada, Pesce2}]\label{Thm:SunadaPesce}
Let $G$ be a finite group acting freely on a closed connected manifold $N$, and let $H_1$ and $H_2$ 
be subgroups of $G$. Then the following are equivalent.
\begin{enumerate}
\item $(G, H_1, H_2)$ is a Gassmann-Sunada triple;
\item for every $G$-invariant metric on $N$, the quotient manifolds 
$H_1\backslash M$ and $H_2 \backslash N$ have the same Laplace spectrum.
\end{enumerate}
\end{thm} 
\noindent
The fact that $(1)$ implies $(2)$ is the well-known theorem of Sunada \cite{Sunada} 
and the reverse implication follows from the work of Pesce \cite{Pesce2}. 

To apply Sunada's method to our manifolds associated to Cayley graphs, suppose
that $H_1$ and $H_2$ are two subgroups of $G$ so that $(G,H_1,H_2)$ is a
Gassmann-Sunada triple.  The theorem above then implies
that $M[H_1\backslash G;S]$ and $M[H_1\backslash G;S]$ are isospectral. 
We can now finish the proof of our main result.

\begin{proof}[Proof of Theorem \ref{Thm:Main}]
The group $G=\Gl_3(\ff_2)$ and the subgroups $H_1$ and $H_2$ of the previous
section form a Gassmann-Sunada triple.  Taking $S=\{A,B\}$, so $n=2$ we see that
$M[V_1;S]$ and $M[V_2;S]$ are connected compact isospectral surfaces of genus
$1+7(g+1)$.

We will now vary the choice of the Riemannian metric on $M$. 
For each $\epsilon> 0$ we choose a metric $g^\epsilon$ on $M$ so that
the handle associated to $A$, thought of as a cylinder, is isometric to
$[0,l_A]\times \R/\epsilon \Z$, and the handle associated to $B$ is isometric to
$[0,l_B]\times \R/\epsilon \Z$, and the diameter of the complement of the two handles 
is below $\epsilon$. 
As $\epsilon$ tends to zero, the metric spaces $(M,g^\epsilon)$
then converge---in the Gromov-Hausdorff sense---to a bouquet
of two circles of circumferences $l_A$ and $l_B$.  

We obtain induced Riemannian metrics $g_i^\epsilon$ on $M[V_i;S]$ for $i=1,2$
and by the remark following Theorem \ref{Thm:SunadaPesce} the manifolds
$(M[V_1;S],g_1^\epsilon)$ and $(M[V_2;S],g_2^\epsilon)$ are isospectral for
every~$\epsilon$.  For both $i=1$ and $i=2$ the spaces
$(M[V_i;S],g_i^\epsilon)$ now converge as $\epsilon \to 0$ to the space
$X_i=H_i\backslash X(l_A,l_B)$ in the Gromov-Hausdorff metric.  By
Gromov-Haussdorff continuity of the covering spectrum (Proposition
\ref{Prop:GHConv}) and the fact that $X_1$ and $X_2$ have distinct covering
spectra, as shown in Proposition \ref{Prop:distinct}, it now follows that for
$\epsilon$ sufficiently small the isospectral surfaces
$(M[V_1;S],g_1^\epsilon)$ and $(M[V_2;S],g_2^\epsilon)$ have distinct covering
spectra as well.  
\end{proof}

The preceding argument actually shows the following.
\begin{thm}
Let $(G, H_1, H_2)$ be a non-trivial Gassmann-Sunada triple, 
and $S\subset G$.
Suppose we can choose lengths $l_s>0$ for all $s\in S$
such that the 
induced length space structures on the geometric realizations of 
the Cayley graphs $H_1\backslash G[S]$ and $H_2\backslash G[S]$
give rise to distinct covering spectra.
Then every connected closed surface $M$ with distinct marked handles
labeled by $S$ has a Riemannian metric so that
$H_1\backslash M[G;S]$ and $H_2\backslash M[V_2;S]$
are isospectral surfaces with distinct covering spectra.
\end{thm}


\bibliographystyle{amsalpha}

\end{document}